\documentclass[11pt]{amsart}
\usepackage{amsfonts,amsmath,latexsym,verbatim,amscd}
\usepackage{amsthm,amsmath,amssymb,amscd}
\usepackage{amssymb}
\usepackage{graphics}
\usepackage{bbm}

\newcommand{\bbC}{\mathbb{C}}

\newcommand{\bbP}{\mathbb{P}}

\newcommand{\bbR}{\mathbb{R}}
\newcommand{\bbS}{\mathbb{S}}

\newcommand{\bbZ}{\mathbb{Z}}

\newcommand{\calI}{\mathcal{I}}

\newcommand{\calO}{\mathcal{O}}

\newcommand{\calS}{\mathcal{S}_\bb}

\newcommand{\calU}{\mathcal{U}}

\newcommand{\HO}{\Hat{O}}

\newcommand{\HS}{\Hat{S}}
\newcommand{\HcS}{\Hat{\mathcal{S}}_\bb}

\newcommand{\HU}{\Hat{U}}
\newcommand{\HV}{\Hat{V}}

\newcommand{\renorm}{\Gamma}
\renewcommand{\Re}{\mathrm{Re}}
\renewcommand{\Im}{\mathrm{Im}}
\renewcommand{\aa}{\mathfrak{a}}
\newcommand{\bb}{\mathfrak{b}}

\newcommand{\unity}{{1\!\!\!\:\mathrm{l}}}
\newcommand{\bloch}{\mathcal{B}}

\newcommand{\pol}{\mathcal{P}_{\aa\bb}}
\newcommand{\reg}{\mathcal{R}_{\aa\bb}}
\newcommand{\ct}{\mathcal{T}_{\aa\bb}}

\newcommand{\banach}[1]{L^{\!#1}}
\newcommand{\ci}{i}
\DeclareMathOperator{\resultant}{resultant}
\DeclareMathOperator{\disc}{Disc}

\newcommand{\R}{{\mathbb R}}

\newcommand{\N}{{\mathbb N}}
\newcommand{\C}{{\mathbb C}}

\newcommand{\tr}{\mathrm{tr}}
\theoremstyle{plain}
\newtheorem{theorem}{Theorem}
\newtheorem*{theorem*}{Theorem}

\newtheorem*{corollary*}{Corollary}

\newtheorem*{proposition*}{Proposition}
\newtheorem{lemma}{Lemma}
\newtheorem*{lemma*}{Lemma}

\newtheorem*{example*}{Example}
\newtheorem{definition}{Definition}
\newtheorem*{definition*}{Definition}

\newtheorem*{notation*}{Notation}
\newtheorem{remark}{Remark}
\newtheorem*{remark*}{Remark}
\title[Singularites of Whitham flow for hyperelliptic curves]{Singularities of Whitham flows for hyperelliptic spectral curves}
\author{L. Hauswirth \and M. Kilian \and M.U. Schmidt}

\address{\tiny L. Hauswirth, Universit Paris-Est, LAMA (UMR 8050), UPEMLV, UPEC, CNRS, F-77454, Marne-la-Valle, France }
\email{hauswirth@univ-mlv.fr}

\address{\tiny M. Kilian, School of Mathematical Sciences, University College Cork, Ireland.}
\email{m.kilian@ucc.ie}

\address{\tiny M.U. Schmidt, Institut f\"ur Mathematik, Universit\"at Mannheim, A5, 6, 68131 Mannheim, Germany.}
\email{schmidt@math.uni-mannheim.de}

\begin{document}
\begin{abstract} 
We consider the Whitham equations for deformations of hyperelliptic spectral curves, which preserve all periods of a meromorphic differential. If the meromorphic differential has a root at a fixed point of the hyperelliptic involution, then the Whitham flow has a singularity. We prove that the stable and unstable manifolds are non-empty and extend the Whitham flow continuously through the singularity.
\end{abstract}
\maketitle
\section{Introduction}
In 1968, Peter Lax~\cite{L} reformulated the KdV equation (see~\cite[(2.1)-(2.2)]{DKN})
$$u_t=6uu_x-u_{xxx}=0$$
for a real or complex function $u$ as
\begin{align}\label{eq:lax}
\partial_t L&=[A,L]&\mbox{where}&&L&=-\partial_x^2+u&\mbox{and}&&A&=4\partial_x^3-3u\partial_x-3\partial_xu.
\end{align}
Writing the KdV equation as a {\it Lax pair} \eqref{eq:lax} makes apparent that the spectrum of the one-dimensional Schr\"odinger operator $L$ is preserved by the time evolution of  $u$. This observation underpinned the development of the {\it spectral curve correspondence}, whereby periodic solutions of the KdV equation are described in terms of spectral curve data. This is comprised of a subvariety of $\mathbb{C}^2$ (the spectral curve), a holomorphic line bundle, a marked point and a meromorphic differential with prescribed poles and periods. Of particular interest are the finite type solutions, for which the spectral curve is hyperelliptic. In particular, the periodic finite type solutions are in one-to-one correspondence to pairs of hyperelliptic curves with holomorphic line bundles. Here we are mainly interested in the moduli spaces of all hyperelliptic curves, which are the spectral curve of a periodic solution of the KdV equation or a similar integrable system. It is now considered a hallmark of an integrable system that it may be formulated as a Lax pair and the spectral curve correspondence has been extended to many other nonlinear partial differential equations. The integrable systems with hyperelliptic spectral curves include besides the KdV equation the nonlinear Schr\"odinger (NLS) equation and the sinh-Gordon equation. For these integrable systems the Whitham equations define vector fields on the corresponding moduli spaces \cite{grinevich-schmidt1995}. These vector fields define on the moduli space local coordinates, which were discovered by Marchenko and Ostrovskii~\cite{MO}. Under special circumstances these vector fields have singularities, and our main result extends their integral curves through these singularities. An important application of this extension of the Whitham flow through singularities is when the finite type solutions account for all solutions. For example all immersed constant mean curvature 2-tori in 3-dimensional space forms, and all properly embedded minimal anuli in $\bbS^2\times\R$ are of finite type. In such cases the integral curves sweep out the moduli spaces, and the Whitham flow can be used to obtain classification results~\cite{hks2, hks3, KSS}.
\subsection {Outline of Paper} 
The primary motivation for studying deformations of a hyperelliptic curve with a meromorphic differential comes from the examples to which our theory applies. Hence we begin in Section~\ref{sec:hyperelliptic} to present the main examples of integrable systems with hyperelliptic spectral curves. First we describe in Section~\ref{sec:kdv} the spectral curves of periodic solutions of the KdV equation. In the two subsequent Sections~\ref{sec:nls}-\ref{sec:sinh} the modifications for periodic solutions of the NLS and the sinh-Gordon equation are explained. In Section~\ref{sec:spectral} we present a general framework for hyperelliptic spectral curves which includes the three foregoing examples. The Whitham flow is introduced in Section~\ref{abdeformation}. We identify the singularities and show in Theorem~\ref{thm:deformation} that the flow is smooth away from them. The two remaining sections investigate the singularities. First we show in Theorem~\ref{smooth param} (Section~\ref{sec:param}) that neighbourhoods around the singularities can be embedded into smooth higher-dimensional moduli spaces. Finally we extend in the main Theorem~\ref{commonroots} (Section~\ref{sec:common root}) the Whitham flow through the singularities.
\section{Integrable systems with hyperelliptic spectral curves}\label{sec:hyperelliptic}
\subsection{Spectral curves of periodic solutions of the KdV equation}\label{sec:kdv}
For periodic solutions $u$ of the KdV equation the spectral curve is the Bloch curve of the Lax operator $L$~\eqref{eq:lax}. Let $T$ be the operator acting as $x$-translation by the period $\gamma$ and note that $L$ and $T$ commute. Then the Bloch curve $\bloch $ is the set of pairs $(\lambda,\mu)\in\mathbb{C}^2 $ such that $L$ and $T$  have a non-trivial common eigenfunction $\psi$ with eigenvalues $\lambda$ and $\mu$, respectively:
\begin{align*}
L\psi(x)&=\lambda\psi(x),&T\psi(x)&=\psi(x+\gamma)=\mu\psi(x).
\end{align*}
Instead of the Lax equation, we use a zero curvature equation introduced by Novikov~\cite{NMPZ} which clarifies the analogies between our three examples. We rewrite the eigenvalue equation as an ODE in terms of $\psi_1=\psi$ and $\psi_2=\psi_x$ (compare~\cite[(2.32)-(2.34)]{DKN})
\begin{align*}\partial_x\big(\begin{smallmatrix}\psi_1\\\psi_2\end{smallmatrix}\big)&=U\big(\begin{smallmatrix}\psi_1\\\psi_2\end{smallmatrix}\big)&\text{with}&&U&=\big(\begin{smallmatrix}0&1\\u-\lambda&0\end{smallmatrix}\big).
\end{align*}
The time evolution can be lifted to an evolution of the eigenfunction:
\begin{align*}\partial_t\big(\begin{smallmatrix}\psi_1\\\psi_2\end{smallmatrix}\big)&=V\big(\begin{smallmatrix}\psi_1\\\psi_2\end{smallmatrix}\big)&\text{with}&&V&=\big(\begin{smallmatrix}u_x&-4\lambda-2u\\4\lambda^2-2u\lambda-2u^2+u_{xx}&-u_x\end{smallmatrix}\big).
\end{align*}
Then the Lax equation~\eqref{eq:lax} is equivalent to
\begin{align}\label{eq:lax2}
\big[\partial_t-V,\partial_x-U\big]&=0.
\end{align}
Now Floquet theory~\cite[Section~3]{Fr} applies with the monodromy $M$:
\begin{align*}
\partial_xF&=UF,&
F|_{x=0}&=\unity,&M(\lambda)&=F|_{x=\gamma}.
\end{align*}
The Bloch curve contains all pairs $(\lambda,\mu)$ with $\mu$ being an eigevalue of $M(\lambda)$. For $u\in\banach{1}(\mathbb{R}/\gamma\mathbb{Z})$ the Peano-Baker series of $F$ converges to an entire function in $\lambda\in\mathbb{C}$. Since $\det F=1$, the Bloch curve is the subvariety
\begin{align}\label{eq:spectral}
&\{(\lambda,\mu)\in\mathbb{C}^2\mid\mu^2-\Delta(\lambda)\mu+1=0\}&
\text{with}&&\Delta(\lambda)&=\tr(M(\lambda)).
\end{align}
In the finite type case the discriminant $\Delta^2-4$ has only finitely many odd order roots. In this case the Bloch curve can be compactified to an algebraic curve $\Sigma$ with a meromorphic function $\lambda$ with a double pole at $\lambda=\infty$, and a function $\mu$ with essential singularity at $\lambda=\infty$. We encode the function $\mu$ into the abelian differential of the second kind $\omega=d\ln\mu$ with second order pole at $\lambda=\infty$. The hyperelliptic algebraic curve $\Sigma$ together with the meromorphic function $\lambda$ of degree two and the Abelian differential of the second kind $\omega$ uniquely determines the corresponding Bloch curve. The corresponding triples $(\Sigma,\lambda,\omega)$ are characterised by the following properties:
\begin{enumerate}
\item[(A)] $\omega$ is anti-symmetric with respect to the hyperelliptic involution $\sigma$. It has one pole of second order, and $\lambda$ has a second order pole there. $\Sigma^\ast$ denotes the complement of this point.
\item[(B)] $\omega$ is the logarithmic derivative of a holomorphic function $\mu$ on $\Sigma^\ast$, which is transformed by the hyperelliptic involution as $\mu\mapsto\mu^{-1}$.
\item[(C)] In case of real potentials, $\Sigma$ is endowed with an anti-holomorphic involution acting as $(\lambda,\mu)\mapsto(\Bar{\lambda},\Bar{\mu}^{-1})$. The fixed point set of this involution has one more component than the geometric genus of $\Sigma$.
\end{enumerate}
As a reference we recommend~\cite[Chapter~2~\S1]{DKN}. The other integrable systems are similar. The main differences are different $\lambda$-dependent matrices $U$ and $V$ and different porperties (A) and (C). In Section~\ref{sec:spectral} we descibe the triples $(\Sigma,\lambda,\omega)$ of periodic finite type potentials by pairs of polynomials $(a,b)$. 
\subsection{Spectral curves of periodic solutions of the NLS equation}\label{sec:nls}
We replace in the foregoing construction of the spectral curves of periodic solutions of the KdV equation, the $\lambda$ dependent matrices $U$ and $V$ by
\begin{align*}U&=\big(\begin{smallmatrix}\ci\lambda&\ci q\\\pm\ci\Bar{q}&-\ci\lambda\end{smallmatrix}\big),&V&=2\lambda U+\big(\begin{smallmatrix}\mp\ci|q|^2&q_x\\\mp\Bar{q}_x&\pm\ci|q|^2\end{smallmatrix}\big)\end{align*}
(compare~\cite[(8.12),(8.13)]{NMPZ}). Here the zero curvature equation is equivalent to the defocusing and the self focusing NLS equation (compare~\cite{Sch}):
$$q_t=-q_{xx}\mp2|q|^2q.$$
The corresponding triples $(\Sigma,\lambda,\omega)$ are characterised by property (B) and
\begin{enumerate}
\item[(A')] $\omega$ is anti-symmetric with respect to the hyperelliptic involution $\sigma$. It has two poles of second order at both poles of first order of $\lambda$ which are interchanged by $\sigma$. $\Sigma^\ast$ denotes the complement of both points.
\item[(C')] In case of real potentials, $\Sigma$ is endowed with an anti-holomorphic involution acting as $(\lambda,\mu)\mapsto(\Bar{\lambda},\Bar{\mu}^{-1})$. The fixed point set of this involution has either one more component than the geometric genus of $\Sigma$ (defocusing NLS) or is empty (self focusing NLS).
\end{enumerate}
\subsection{Spectral curves of periodic solutions of the sinh-Gordon equation}\label{sec:sinh}
There exist different versions of the sinh-Gordon equation \cite{bab}. The highest derivatives might be either the Laplace operator or the wave operator and the constants differ. We present here the version of the sinh-Gordon equation in our main application~\cite{hks1,hks2}, as derived from the Gau{\ss}-Codazzi equations of constant mean curvature surfaces in three-dimensional space forms \cite{PS,Hi,bob}. Instead of the variable $t$ we use here the variable $y$:
$$
(\partial_x^2+\partial_y^2)u + 2\sinh(2u)= 0.
$$
This equation is equaivalent to~\eqref{eq:lax2} with $t$ replaced by $y$ and with the following $\lambda$-dependent matrices $U$ and $V$:
\begin{align*}
U&=\tfrac{1}{2}\big(\begin{smallmatrix}-\ci u_y &\ci e^{-u}+\ci\lambda^{-1}e^u\\\ci e^{-u}+\ci\lambda e^u&\ci u_y\end{smallmatrix}\big),&V&=\tfrac{1}{2}\big(\begin{smallmatrix}\ci u_x &e^{-u}-\lambda^{-1}e^u\\-e^{-u}+\lambda e^u& -\ci u_x\end{smallmatrix}\big).
\end{align*}
The corresponding triples $(\Sigma,\lambda,\omega)$ are characterised by property (B) and
\begin{enumerate}
\item[(A'')] $\omega$ is anti-symmetric with respect to the hyperelliptic involution $\sigma$. It has two poles of second order at the unique pole and root of second order of $\lambda$. $\Sigma^\ast$ denotes the complement of both points.
\item[(C'')] In case of real potentials, $\Sigma$ is endowed with an anti-holomorphic involution acting as $(\lambda,\mu)\mapsto(\Bar{\lambda},\Bar{\mu})$ without fixed points.
\end{enumerate}
\section{Hyperelliptic spectral curves}\label{sec:spectral}
\subsection{Hyperelliptic spectral curves} We consider $\Sigma$ with a meromorphic function $\lambda$ of degree two and with either one, or two or four marked points. We distinguish between four cases parameterized by $(\aa,\bb)\in\{0,1\}^2$:\vspace{2mm}

$(\aa,\bb)=(0,0)$ corresponds to the KdV equation in Section~\ref{sec:kdv},

$(\aa,\bb)=(1,0)$ corresponds to the NLS equation in Section~\ref{sec:nls},

$(\aa,\bb)=(0,1)$ corresponds to the sinh-Gordon equation in Section~\ref{sec:sinh},

$(\aa,\bb)=(1,1)$ corresponds to non-conformal harmonic maps to $\bbS^3$ \cite{Hi}.\vspace{2mm}

For any subset $S\subset\bbP^1$ let $\HS$ denote the set
\begin{equation}\label{eq:preimage}
\HS=\{p\in\Sigma\mid\lambda(p)\in S\}.
\end{equation}
The marked points are invariant under the hyperelliptic involution $\sigma$ and fit togehter to form a set $\HcS$ of the form~\eqref{eq:preimage}. It is determined by the parameter $\bb$:
\begin{align}\label{eq:S}
\calS&=\begin{cases}\{\infty\}&\mbox{for }\bb=0\\\{0,\infty\}&\mbox{for }\bb=1.\end{cases}
\end{align}
The triples $(\Sigma,\lambda,\omega)$ are described by a pair of polynomials $(a,b)$ of degree $2g+(1-\aa)(1-\bb)$ and $g+\bb+\aa\bb$, respectively, which do not vanish in $\C\cap\calS$. The polynomial $a$ defines the variety $\Sigma^\ast$ with hyperelliptic involution $\sigma$:
\begin{align}\label{eq:spectral curve}
\Sigma^\ast&\!=\!\{(\lambda,\nu)\!\in\!(\bbP^1\!\setminus\!\calS)\times\bbC\mid\nu^2\!=\!\lambda^{(1-\aa)\bb}a(\lambda)\},&\sigma\!:\!(\lambda,\nu)\mapsto\!(\lambda,-\nu).
\end{align}
The spectral curve $\Sigma$ is the unique projective (not necessarily smooth) hyperelliptic curve with finitely many smooth points $\HcS$, such that $\Sigma\setminus\HcS$ is biregular to $\Sigma^\ast$. The preimage of each point of $\calS$ contains $2^\aa$ points, and $\HcS$ contains $2^{\aa+\bb}$ points. The polynomial $b$ defines the meromorphic differential
\begin{align}\label{eq:def b}
\omega&=\frac{b(\lambda)}{\lambda^{\bb+\aa\bb}\nu}d\lambda
\end{align}
without residues at $\HcS$. $\Sigma$ and $\Sigma^\ast$ are endowed with an anti-holomorphic involution, and the polynomials $(a,b)$ satisfy a reality condition, i.e.\ are fixed points of an anti-linear involution of the space of polynomials of degree $2g+(1-\aa)(1-\bb)$ and $g+\bb+\aa\bb$, respectively. We consider three involutions:
\begin{align}\label{eq:involution 1}
(\lambda,\nu,a(\lambda),b(\lambda))&\mapsto(\Bar{\lambda},\Bar{\nu},\overline{a(\Bar{\lambda})},\overline{b(\Bar{\lambda})})\\\label{eq:involution 2}
(\lambda,\nu,a(\lambda),b(\lambda))&\mapsto(\Bar{\lambda}^{-1}\hspace{-1mm},\Bar{\lambda}^{g+\aa}\Bar{\nu},\lambda^{\deg a}\overline{a(\Bar{\lambda}^{-1})},-\lambda^{\deg b}\overline{b(\Bar{\lambda}^{-1})})\\\label{eq:involution 3}
(\lambda,\nu,a(\lambda),b(\lambda))&\mapsto(-\Bar{\lambda}^{-1}\hspace{-1mm},\Bar{\lambda}^{g+\aa}\Bar{\nu},\lambda^{\deg a}\overline{a(-\Bar{\lambda}^{-1})},(-1)^g\lambda^{\deg b}\overline{b(-\Bar{\lambda}^{-1})}).
\end{align}
For pairs $(a,b)$ in the fixed point set of~\eqref{eq:involution 1} or~\eqref{eq:involution 2} or~\eqref{eq:involution 3}, $\Sigma$ is invariant and $\omega$ is mapped to $\Bar{\omega}$. For $\bb=0$ only~\eqref{eq:involution 1} is considered, since the involution should permute the marked points. For $\bb=1$ all involutions~\eqref{eq:involution 1}-\eqref{eq:involution 3} are compatible. However, the involution~\eqref{eq:involution 3} has non-trivial fixed points $(a,b)$ only if both degrees $2g$ and $g+1+\aa$ are even, i.e.\ $g$ is odd for $\aa=0$ and $g$ is even for $\aa=1$. Furthermore, the roots of all fixed points $(a,b)$ of~\eqref{eq:involution 2} or \eqref{eq:involution 3} are contained in $\C\setminus\calS$. For $(\aa,\bb)=(0,1)$ the corresponding equations for all three involutions are described in~\cite{bab}. For $(\aa,\bb)=(1,1)$ to the authors' knowledge only the involution~\eqref{eq:involution 2} has been considered so far in~\cite{Hi}.

We normalise the polynomials $a$ in such a way, that they are essentially determined by their roots. In case of the involution~\eqref{eq:involution 1} the normalised poynomials $a$ have highest coefficient $1$ and $a$ is uniquely determined by the roots. In case of the two other involutions~\eqref{eq:involution 2} or~\eqref{eq:involution 3} the highest and lowest coefficients have absolute value $1$. In these cases the polynomials $a$ are determined by the roots only up to multiplication with $\pm 1$. In particular, a continous deformation of a given normalised polynomial $a$ is uniquely determined by the corresponding deformations of the roots of $a$. In some cases there is a unique way to determine $a$ in terms of its roots. For example, in case of our main application~\cite{hks1,hks2} with $(\aa,\bb)=(0,1)$ and $\calS=\{0,\infty\}$, in addition the involution~\eqref{eq:involution 2} is assumed to have no fixed points on $\Sigma$. Then $a$ is uniquely determined by its roots and the condition
\begin{align*}
\lambda^{-g-1}\,a(\lambda)&\leq 0&\mbox{for }\lambda\in\bbS^1.
\end{align*}
\begin{definition}
For $(\aa,\bb)\in\{0,1\}^2$ let $\pol^g$ be the space of all pairs $(a,b)$ with a normalised polynomial $a$ of degree $2g+(1-\aa)(1-\bb)$ and a polynomial $b$ of degree $g+\bb+\aa\bb$ with the following three properties: Both $a$ and $b$ have no roots in $\C\cap\HcS$, they are fixed points of the involution~\eqref{eq:involution 1} or~\eqref{eq:involution 2} or~\eqref{eq:involution 3}, and finally~\eqref{eq:def b} has no residues at the elements of $\HcS$.

The subset of $(a,b)\in\pol^g$ with $b/a$ having at most first order poles on $\C\setminus\calS$ is denoted by $\reg^g$.

The subset of $\{(a,b)\in\reg^g\mid\resultant(a,b)\ne 0\}$ is denoted by $\ct^g$.
\end{definition}
Since $\omega$ is antisymmetric with respect to $\sigma$, for $\aa=0$, $\omega$ has no residues at the fixed points $\HcS$ of $\sigma$ for all $a$ and $b$. For $\aa=1$ this condition reduces the real dimension of the affine space of pairs $(a,b)$ by $1+\bb$. Hence $\pol^g$ is an open subset of an affine real space of dimension
\begin{equation}\label{eq:dim}
2g+(1-\aa)(1-\bb)+g+\bb+\aa\bb+1-\aa(1+\bb)=3g+2-2\aa+\aa\bb.
\end{equation}
The subset $\reg^g\subset\pol^g$ is neither open nor closed. The condition that $b/a$ has at most first order poles on $\C\setminus\calS$ is equaivalent to $\omega$~\eqref{eq:def b} beeing the derivative of a function which is regular on $\Sigma^\ast$. For $(a,b)\in\reg^g$ higher order roots of $a$ are common roots of $a$ and $b$. Hence $\ct^g$ is the open subset of $(a,b)\in\pol^g$ with non-vanishing discriminant $\disc(a)$ and non-vanishing $\resultant(a,b)$. Hence $\ct^g$ is a submanifold of the real affine space $\pol^g$.
\subsection{M\"obius transformations}\label{sec:moebius} All M\"obius transformations of $\lambda$ which preserve $\calS$ and the corresponding involution~\eqref{eq:involution 1}-\eqref{eq:involution 3} change the polynomials $(a,b)$, but the corresponding spectral curves $\Sigma$ are biholomorphic and the forms $\omega$ are the pullbacks of each other. For $\bb=0$ and the involution~\eqref{eq:involution 1} they take the form $\lambda\mapsto\alpha\lambda+\beta$ with $0\ne\alpha,\beta\in\R$. For $\bb=1$ they take the form $\lambda\mapsto\alpha\lambda$ with $\alpha\in\R\setminus\{0\}$ in case of the involution~\eqref{eq:involution 1} and $|\alpha|=1$ in case of the involution~\eqref{eq:involution 2} or~\eqref{eq:involution 3}. Let us now introduce real parameters, whose values parameterize these M\"obius transformations. For this purpose we define $\gamma\in\C$ such that $\omega-\gamma d(\nu\lambda^{\aa-g})$ is holomorphic at $\lambda=\infty$. Since $\omega$ and $d(\nu\lambda^{\aa-g})$ both have a second order pole without residue at the points in $\HcS$ over $\lambda=\infty$, $\gamma$ exists and is unique. For $\bb=0$ we only consider the involution~\eqref{eq:involution 1} and $\gamma$ is real. In this case the two-dimensional M\"obius transformations are parameterized by $\gamma$ and the second highest coefficient of $a$. For $\bb=1$ and the involution~\eqref{eq:involution 1} the one-dimensional M\"obius transformations are parameterized by $\gamma$. Finally for $\bb=1$ and the involutions~\eqref{eq:involution 2} or \eqref{eq:involution 3} $\gamma/|\gamma|\in \bbS^1$ parameterizes the one-dimensional M\"obius transformations.
\subsection{The isoperiodic set}\label{sec:isoperiodic} Let us now introduce our main object, namely the isoperiodic set. For $(a,b)\in\reg^g$ the integrals of $\omega$ along cyclels in $H_1(\Sigma,\bbZ)$ are called periods. The isoperidoic set is the subset of $\reg^g$ on which all these periods are locally constant. To make this more precise we calculate the periods along special representatives of the cycles. For $\aa=1$ we call the elements of $\{\lambda\in\C\setminus\calS\mid\lambda^{(1-\aa)\bb}a(\lambda)=0\}$ branch points of $a$ and for $\aa=0$ the elements of $\{\lambda\in\C\mid a(\lambda)=0\}\cup\{\infty\}$. For a smooth path in $\bbP^1$ which connects two branch points and does not pass through any other branch point, the difference of the two lifts to $\Sigma$ defines a cycle in $H_1(\Sigma,\bbZ)$. Such cycles generate $H_1(\Sigma,\bbZ)$. Because $\omega$ is anti-symmetric with respect to the hyperelliptic involution $\sigma$, the integrals of $\omega$ along such a cycle is twice the integral of $\omega$ along one of the two lifts to $\Sigma$ of the path in $\bbP^1$. For $\aa=0$ the poles of $\omega$ are branch points and the integral of $\omega$ along a path starting or ending at a pole of $\omega$ does not exist. In this case we choose a meromorphic function $g$ on $\bbP^1$ without poles at the branch points in $\C\setminus\calS$, such that $\omega-d(g\nu)$ has no poles at $\calS$. Since all periods of $d(g\nu)$ vanish the integral of $\omega-d(g\nu)$ along a path from a branch point in $\calS$ to another branch point is equal to half the period of $\omega$ along the corresponding cycle. Here the path should not pass through other branch points or other poles of $g$. In the following definition the integral of $\omega$ along such paths ending at poles of $\omega$ should be replaced by the correponding integral of $\omega-d(g\nu)$ with a meromophic function $g$ as decribed above. In order to define the isoperiodic set in the neighbourhood of a given $(a,b)\in\reg^g$ we choose a covering $\calO$ of $\bbP^1$ by simply connected open subsets $O\in\calO$ with the following properties:
\begin{enumerate}
\item[(i)] $\Bar{O}$ contains for each $O\in\calO$ exactly one root of $a$.

\item[(ii)] $O\cap U$ is either empty or connected for any two $O,U\in\calO$.
\end{enumerate}
Since $\calO$ is a covering the branch point in (i) belongs to $O$. Such a covering defines an open neighbourhood $\calU$ of $(a,b)$ in $\reg^g$: The set of all $(\Tilde{a},\Tilde{b})\in\reg^g$ such that $O$ and $\bar{O}$ contain for all $O\in\calO$ the same number of roots of $\tilde{a}$ as roots of $a$ (counted with multiplicity). We decorate the functions $\lambda$, $\nu$, the form $\omega$ and the spectral curve $\Sigma$ corresponding to $(\Tilde{a},\Tilde{b})$ with a tilde.
\begin{definition}\label{isoperiodic}
Let $(a,b)\in\reg^g$ and $\calO$ be a covering of $\bbP^1$ by simply connected open subsets of $\bbP^1$ obeying~(i)-(ii). The isoperiodic set $\calI(a,b)$ of $(a,b)$ is the set of all $(\Tilde{a},\Tilde{b})\in\calU$ with two properties:
\begin{enumerate}
\item[(i)] For all $O\in\calO$ the integrals of $\Tilde{\omega}$ along a path in $O$ connecting two branch points of $\Tilde{a}$ in $O$ vanish.
\item[(ii)] For all pairs of non-disjoint $O\ne U\in\calO$ the integrals of $\Tilde{\omega}$ along a path in $O\cup U$ from a branch point of $\Tilde{a}$ in $O$ to a branch point of $\Tilde{a}$ in $U$ is equal to the integral of $\omega$ along a path in $O\cup U$ from the unique branch point of $a$ in $O$ to the unique branch point of $a$ in $U$.
\end{enumerate}
\end{definition}
If $\calO$ has for given $(a,b)\in\reg^g$ the properties~(i)-(ii), then $\HO$ is for all $O\in\calO$ simply connected and there exist unique meromorphic functions
\begin{align}\label{eq:local}
f_O\!:\!O&\!\to\!\bbP^1\!,&\!\xi_O&\!=\!f_O\nu&\!\mbox{with}\!&&d\xi_O&\!=\!\frac{2f'_O(\lambda)a(\lambda)\!-\!f_O(\lambda)a'(\lambda)}{2\nu}d\lambda\!=\!\omega|_{\HO}.
\end{align}
Furthrmore, if $(\Tilde{a},\Tilde{b})\in\calU$ has property~(i) in Definition~\ref{isoperiodic}, then there exist unique meromorphic functions
\begin{align}\label{eq:local 2}
\Tilde{f}_O\!:\!O&\!\to\!\bbP^1\!,&\!\Tilde{\xi}_O&\!=\!\Tilde{f}_O\Tilde{\nu}&\!\mbox{with}\!&&d\Tilde{\xi}_O&\!=\!\frac{2\Tilde{f}'_O(\lambda)\Tilde{a}(\lambda)\!-\!\Tilde{f}_O(\lambda)\Tilde{a}'(\lambda)}{2\Tilde{\nu}}d\lambda\!=\!\Tilde{\omega}|_{\Tilde{O}}.
\end{align}
Here $\Tilde{O}$ denotes the preimage of $O$ with respect to the map $\Tilde{\lambda}:\Tilde{\Sigma}\to\bbP^1$ of the spectral curve of $\Tilde{a}$ analogous to~\eqref{eq:preimage}. Then condition~(ii) in Definition~\ref{isoperiodic} is equivalent to the following equations for all non-disjoint $O\ne U\in\calO$
\begin{equation}\label{eq:const}
\xi_U-\xi_O=\Tilde{\xi}_U-\Tilde{\xi}_O.
\end{equation}
Note that both sides are by definition of $\xi_O$, $\xi_U$, $\Tilde{\xi}_O$ and $\Tilde{\xi}_U$ constant on $\HO\cap\HU$ and $\Tilde{O}\cap\Tilde{U}$, repsectively. Moreover, the isoperiodic sets defined in terms of two coverings with the proerties~(i)-(ii) conincide on the intersection of the correpsonding open neighbourhoods of $(a,b)$ in $\reg^g$.
\subsection{Adding and removing double points}\label{sec:double points}
We can add at any $\lambda_0$ to $a$ a double root and to $b$ a simple root without changing $\omega$. Since $a$ and $b$ are fixed under the involution~\eqref{eq:involution 1} or~\eqref{eq:involution 2} or~\eqref{eq:involution 3}, we should add either two roots interchanged by this involution or a root in the fixed point set. However, $b/a$ can have at most first order poles on $\bbP^1\setminus\HcS$ and the local functions function $f_O$~\eqref{eq:local} are regular on $\Sigma^\ast$. Hence we can add only at those $\lambda_0\in\bbC\setminus\calS$ double roots of $a$ and simple roots of $b$, where the corresponding $f_O$ vanishes. More generally, if for a pair $(a,b)\in\reg^g$ an appropriate polynomial $p$ divides the functions $f_O$ (i.e.\ $f_O/p$ are holomorphic on $O\setminus\calS$), then the pair $(\Hat{a},\Hat{b})=(ap^2,bp)$ belongs to $\reg^{g+\deg p}$. Here the polynomial $p$ should be normalised and fixed under the involution~\eqref{eq:involution 1} or~\eqref{eq:involution 2} or~\eqref{eq:involution 3}.

Conversely, if the first polynomial of $(\Hat{a},\Hat{b})\in\reg^{\Hat{g}}$ has higher order roots in $\bbP^1\setminus\calS$, then also $\Hat{b}$ has a root there, such that the local functions $f_O\nu$ are holomorphic there. Hence for an appropriate polynomial $p$ whose square divides $\Hat{a}$ the polynomial $p$ divides $\Hat{b}$ and the pair $(a,b)=(\Hat{a}/p^2,\Hat{b}/p)$ belongs to $\reg^{\Hat{g}-\deg p}$. Again the polynomial $p$ should be normalised and fixed under the involution~\eqref{eq:involution 1} or~\eqref{eq:involution 2} or~\eqref{eq:involution 3}. If $p$ collects all higher order roots of $\Hat{a}$, then the polynomial $a$ of the transformed pair $(a,b)$ has only simple roots. We denote these transformations by
\begin{align}\label{eq:double roots}
(a,b)&\mapsto(\Hat{a},\Hat{b})=(ap^2,bp)&\text{and}&&
(\Hat{a},\Hat{b})&\mapsto(a,b)=(a/p^2,b/p).
\end{align}
The isoperiodic set of such pairs $(\Hat{a},\Hat{b})$ with higher order roots of $a$ contains generically pairs whose first polynomial has only simple roots and the geometric genus is increased. In case of an involution~\eqref{eq:involution 1} or~\eqref{eq:involution 2} without fixed points and a double point interchanged by the hyperelliptic involution (i.e.\ a fixed point of the composition of the anti-holomorphic involution with the hyperlliptic involution) only on a subsest of the level set the condition is preserved that the anti-holomorphic involution has no fixed points. Such data $(a,b)$ are boundary points of the subsest of level set in $\reg^g$ whose elements correspond to $\Sigma$ without fixed points of the involution~\eqref{eq:involution 1} or~\eqref{eq:involution 2}.
%
%
%
%
\section{Deformations of hyperelliptic spectral curves}\label{abdeformation}
In this section we describe one-dimensional families $t\mapsto(a,b)$ parameterized by a real variable  $t\in(-\epsilon,\epsilon)$ in $\calI((a,b)|_{t=0})$, they are integral curves of vector fields on $\ct^g$.
\subsection{The Whitham flow}\label{sec:whitham}
We consider such a smooth family $t\mapsto(a,b)$ parameterized by $t\in(\-\epsilon,\epsilon)$. The functions on the corresponding spectral curves depend on $\lambda$ and $t$. The derivative with respect to $t$ is denoted by a dot. On the open set $O\in\calO$ of a covering as in Section~\ref{sec:isoperiodic} we have
\begin{align}\label{eq:dot}
\frac{\partial}{\partial t}\xi_O&=\dot{f}_O(\lambda)\nu+f_O(\lambda)\dot{\nu}=
\frac{2\dot{f}_O(\lambda)a(\lambda)+f_O(\lambda)\dot{a}(\lambda)}{2\nu}.
\end{align}
Since the left-hand side of~\eqref{eq:const} does not depend on $t$, the left-hand sides of~\eqref{eq:dot} fit together to form a global meromorphic function on $\Sigma^\ast$ with poles of at most first order at the points in $\HcS$. This function is of the form
\begin{equation}\label{eq:def c}
\frac{c(\lambda)}{\lambda^{\aa\bb}\nu}
\end{equation}
with a polynomial $c$ of degree at most $g+1+\aa\bb$ in the fixed point set of an involution selected out of the following list accordingly to~\eqref{eq:involution 1} or~\eqref{eq:involution 2} or~\eqref{eq:involution 3}:
\begin{align}\label{eq:involutions c}
c(\lambda)&\!\mapsto\!\overline{c(\Bar{\lambda})}\mbox{ or}&\hspace{-2mm}c(\lambda)&\!\mapsto\!\lambda^{g+1+\aa\bb}\overline{c(\Bar{\lambda}^{-1})}\mbox{ or}&\hspace{-2mm}c(\lambda)&\!\mapsto\!(\!-1)^{\aa}\lambda^{g+1+\aa\bb}\overline{c(\!-\Bar{\lambda}^{-1})}.
\end{align}
Note that the third involution again has only non-trivial fixed points for $g+1+\aa$ even with $(-1)^{\aa}=-(-1)^g$. For $\bb=0$ the choices $c(\lambda)=b(\lambda)$ and $c(\lambda)=\lambda b(\lambda)$ correspond to the infinitesimal M\"obius transformations $\dot{\lambda}=1$ and $\dot{\lambda}=\lambda$. For $\bb=1$ there is only one such trivial choice $c(\lambda)=\sqrt{\pm 1}b(\lambda)$ with $+$ for~\eqref{eq:involution 1} and $-$ for~\eqref{eq:involution 2} or~\eqref{eq:involution 3}. It corresponds to the infinitesimal M\"obius transformation $\dot{\lambda}=\sqrt{\pm 1}\lambda$. We differentiate \eqref{eq:def b} with respect to $t$, and~\eqref{eq:def c} with respect to $\lambda$ and derive
\begin{align*}
\frac{\partial}{\partial\lambda}\frac{c(\lambda)}{\lambda^{\aa\bb}\nu}&=
\frac{2a(\lambda)c'(\lambda)-a'(\lambda)c(\lambda)}{2\lambda^{\aa\bb}a(\lambda)\nu}-\frac{\aa\bb c(\lambda)}{\lambda^{1+\aa\bb}\nu},\\
\frac{\partial}{\partial t}\frac{b(\lambda)}{\lambda^{\bb+\aa\bb}\nu}&=\frac{\dot{b}(\lambda)}{\lambda^{\bb+\aa\bb}\nu}-\frac{b(\lambda)\dot\nu}{\nu^2\lambda^{\bb+\aa\bb}}=\frac{2a(\lambda)\dot{b}(\lambda)-b(\lambda)\dot{a}(\lambda)}{2\lambda^{\bb+\aa\bb}a(\lambda)\nu}.
\end{align*}
Hence second partial derivatives of $\xi_O$~\eqref{eq:local} commute if and only if
\begin{align}\label{eq:integrability 1}
\lambda^{1-\bb}\big(2a(\lambda)\dot{b}(\lambda)-\dot{a}(\lambda)b(\lambda)\big)&=\lambda\big(2a(\lambda)c'(\lambda)-a'(\lambda)c(\lambda)\big)-2\aa\bb a(\lambda)c(\lambda).
\end{align}
For $\bb=0$ both sides vanish at $\lambda=0$. Hence~\eqref{eq:integrability 1} is an equation for the polynomial $2a\dot{b}-\dot{a}b$ in the fixed point set of the action of the involution~\eqref{eq:involution 1} or~\eqref{eq:involution 2} or~\eqref{eq:involution 3} on this polynomial. Hence the number of independent real equations in~\eqref{eq:integrability 1} is at most $1+\deg a+\deg b$, which is the real dimension of $\pol^g$, since $a$ is normalised. If $a$ and $b$ have no common root, then~\eqref{eq:integrability 1} uniquely determines the Taylor coefficients of $\dot{a}$ and $\dot{b}$ at the roots of $a$ and $b$ up to the orders of the roots minus one, respectively, and the highest coefficient of $\dot{b}$. Hence in this case $c$ uniquely determines the derivatives $(\dot{a},\dot{b})$ of the pair $(a,b)$ with normalised polynomial $a$. Moreover, smooth $c$ define unique $(\dot{a},\dot{b})$ depending smoothly on $(a,b)$. This proves the first statement of the following theorem:
\begin{theorem}\label{thm:deformation}
For $(a,b)\in\ct^g$ and a polyonmial $c$ of degree $g+1+\aa\bb$ in the fixed point set of the involution selected out of the list~\eqref{eq:involutions c} accordingly to the involution~\eqref{eq:involution 1} or~\eqref{eq:involution 2} or~\eqref{eq:involution 3} the equation~\eqref{eq:integrability 1} determines a unique tangent vector $(\dot{a},\dot{b})\in T_{(a,b)}\ct^g$. All elements in the intersection of the kernels of the derivatives of all periods are of this form, and $\calI(a,b)\cap\ct^g$ is a $(g+2+\aa\bb)$-dimensional submanifold of $\ct^g$.
\end{theorem}
\begin{proof}
A tangent vector $(\dot{a},\dot{b})$ belongs to the intersection of the kernels of the derivatives of all periods, if and only if the infinitesimal version of the conditions~(i) and (ii) in Definition~\ref{isoperiodic} are satisfied. In particular, they are the tangent vectors determined by the polynomials $c$. This shows that the kernel is $g+2+\aa\bb$-dimemsional. The rank of $H_1(\Sigma,\bbZ)$ is equal to the number of branch points minus $2$, which is $2g-2\aa$. Hence the dimension~\eqref{eq:dim} of $\ct^g$ minus the number of independent periods is equal to the dimension $g+2+\aa\bb$ of the former kernel.  Hence the derivatives of the periods are linear independent, and $\calI(a,b)\cap\ct^g$ is due to the implicit function theorem a submanifold of $\ct^g$ of dimension $g+2+\aa\bb$. The former tangent vectors determined by the polynomials $c$ are the elements of the tangent space of this submanifold.
\end{proof}
In the case, where $a$ and $b$ have only simple pairwise different roots $\alpha_i$ and $\beta_j$, respectively, we get the equations
\begin{align}\label{eq:integrability 2}
\dot{\alpha}_i&=-\frac{\dot{a}(\alpha_i)}{a'(\alpha_i)}=-\frac{\alpha_i^{\bb}c(\alpha_i)}{b(\alpha_i)},&\dot{\beta_j}&=-\frac{\dot{b}(\beta_j)}{b'(\beta_j)}=\frac{\beta_j^{\bb}\big(a'(\beta_j)c(\beta_j)\!-\!2a(\beta_j)c'(\beta_j)\big)}{2a(\beta_j)b'(\beta_j)}\!+\!\frac{\aa\bb c(\beta_j)}{\beta_j^{1-\bb}b'(\beta_j)}.
\end{align}
\subsection{The singularities of the Whitham flow}
For given $(a,b)\in\pol^g$ the equations~\eqref{eq:integrability 1} are linear equations for the coefficients of $(\dot{a},\dot{b})$ depending linearly on the coefficients of $c$. Such equations have a general solution, if the determinant of the corresponding matrix of coefficients does not vanish. This determinant is a polynomial with respect to the coefficients of $a$ and $b$. Due to Theorem~\ref{thm:deformation}, this determinant does not vanish for $\disc(a)\ne0\ne\resultant(a,b)$. Hence~\eqref{eq:integrability 1} defines a vector field on $\pol^g$, which depends rationally on the coefficients of $a$ and $b$, and linearly on the coefficients of $c$.

We may multiply these meromorphic vector fields with the least common multiple of the denominators of all entries of the rational vector field and obtain a non-trivial holomorphic vector field. The integral curves of the latter holomorphic vector field are reparameterized integral curves of the former meromorphic vector fields. In particular, the integral curves of both vector fields have the same image sets in $\pol^g$, in which we are mainly interested. Hence we may consider the holomorphic vector fields instead of the meromorphic vector fields. Instead of the zeroes of the denominators of meormorphic vector fields defined by~\eqref{eq:integrability 1} the holomorphic vector fields have zeroes. They are called singularities of the holomorphic vector fields, since in finite time the trajectories can neither pass in or out from the roots of a holomorhic vector field. However, the linearisation of the vector fields at the roots give some information on the stable and unstable manifolds, i.e. the trajectories passing in and out in infinite time. Moreover, even if the linearisation of the vector field vanishes, then the linearisation of a blow up may be non-trivial. Along these lines we investigate the holomorphic vector fields defined by~\eqref{eq:integrability 1} in Section~\ref{sec:common root}.

In the next section we deform the local functions $\xi_O$ instead of $(a,b)$ and the derivatives $\omega$ of the $\xi_O$. The Whitham flow turns out to be non-singular at $(a,b)\in\reg^g$, if $a$ and the $f_O$'s have no common roots.
\section{A smooth parametrisation of the isoperiodic set}\label{sec:param}
We consider data $(a,b)\in\reg^g$ with a covering $\calO$ as in Section~\ref{sec:isoperiodic} and the corresponding functions $f_O$ and $\xi_O$~\eqref{eq:local}. In this section we show that $\calI(a,b)$ is a submanifold of $\pol^g$ at $(a,b)\in\reg^g$, if there are no common roots of the $f_O$'s and $a$. Due to Section~\ref{sec:double points} this condition is equivalent to $b/a$ having simple poles at common roots of $a$ and $b$. This construction allows us to increase the genus of a spectral curve $\Sigma$ by opening a double point. If there are common roots of $a$ and one of the $f_O$'s, then we consider spectral data $(\Hat{a},\Hat{b})=(ap^2,bp)$, such that the corresponding functions $\Hat{f}_O=f_O/p$~\eqref{eq:local} do not have common roots with $\Hat{a}$. This means that we add to $a$ as many double roots as the order of the corresponding root of $f_O$. This procedure reduces the order of the root of $f_O$. Therefore the number $n$ of double points we can add to any root of $f_O$ for given $(a,b)\in\reg^g$ is equal to the order of the root of $f_O$ and the resulting $(\Hat{a},\Hat{b})\in\reg^{g+n}$ will not have a root of $\Hat{f}_O=f_O/p$ at the root of $f_O$. Moreover we have the choice to do the same simultaneously at finitely many roots of $f_O$. The number of roots of $f_O$ is locally constant. In particular, in $\calI(a,b)$ a neighbourhood of any $(a,b)\in\reg^g$ is embedded into the manifold $\calI(\Hat{a},\Hat{b})$. In the next section we will see that we can pass in $\calI(\Hat{a},\Hat{b})$ through a singularity of the vector field described in Theorem \ref{thm:deformation}, when $a$ and $b$ have common roots without increasing the geometric genus.

Now we construct the smooth parametrisation of $\calI(a,b)$ if the $f_O$'s and $a$ have no common roots. We choose simply connected disjoint neighbourhoods $V_1,\ldots,V_M$ in $\C\setminus\calS$ of all roots of $b$ including the common roots with $a$. Let $\HV_1,\ldots,\HV_M$ denote the preimages~\eqref{eq:preimage} in $\Sigma^\ast$. We assume that each $\HV_m$ is contained in the domain $\HO$ of one of the functions $\xi_O$ and set $\xi_m=\xi_O|_{\HV_m}$. Since $\sigma^*\xi_m=-\xi_m$, the square depends only on $\lambda$ (\cite[Theorem~8.2]{Fo})
\begin{equation}\label{eq:local description}
\xi_m^2=A_m
\end{equation}
with holomorphic functions $A_m:V_m\to\C$ which vanish at the roots of $a$. We choose $V_m$ sufficiently small, such that the derivative of $A_m$ has no roots besides the corresponding root of $b$, which is also a root of $d\xi_m=\omega|_{\HV_m}$. For small enough $V_m$ there exists a biholomorphic map $\lambda\mapsto z_m(\lambda)$ from $V_m$ to a simply connected open neighbourhood $W_m$ of $0\in\mathbb{C}$, such that
\begin{equation}\label{eq:pol A}
A_m(\lambda)=z_m^{d_m}(\lambda)+ a_m.
\end{equation}
At a root of $b$ which is not a root of $a$ (i.e.\ a root of $\omega$), the constant is $a_m\in\mathbb{C}$, and $d_m -1$ is the order of the root of $b$. At a common root of $a$ and $b$, the constant is $a_m =0$ (this includes the case of double points).

We  describe  spectral curves in a neighbourhood of the given
spectral curve by small perturbations $\Tilde{A}_1,\ldots,\Tilde{A}_M$
of the polynomials $A_1,\ldots,A_M$. More precisely, we consider
polynomials $\Tilde{A}_1,\ldots,\Tilde{A}_M$ of the form
\begin{equation}\label{eq:deformed pol A}
\Tilde{A}_m(z_m)=z_m^{d_m}+\Tilde{a}_{m,1}z_m^{d_m-1}+\Tilde{a}_{m,2}z_m^{d_m-2}+\ldots+\Tilde{a}_{m,m}
\end{equation}
with coefficients $\Tilde{a}_{m,2},\ldots,(\Tilde{a}_{m,m}-a_m)$ near zero. By a shift $z\mapsto z+z_0$, we can always assume that the sum of the roots is zero and then $\Tilde{a}_{m,1}=0$. We glue each $W_m$ of the sets $W_1,\ldots,W_M$ to $\mathbb{P}^1\setminus(V_1\cup\ldots\cup V_M)$ in such a way along the boundary of $V_m$ that for all $m=1,\ldots,M$  the polynomial $\Tilde{A}_m $ coincides with the unperturbed function $A_m $ in a tubular neighborhood of the boundary $\partial W_m$. We obtain a new copy of $\bbP^1$. By uniformisation, there exists a new global parameter $\Tilde{\lambda}$, which takes at the points of $\calS$ the same values as $\lambda$. This new parameter is unique up to a M\"obius transformation decsribed in Section~\ref{sec:moebius}. In particular there exist biholomorphic maps $\lambda\mapsto\Tilde{\lambda}=\phi(\lambda)$ from $\lambda\in\bbP^1\setminus(V_1\cup\ldots\cup V_M)$ and $z_m\mapsto\Tilde{\lambda}=\phi_m(z_m)$ from $z_m\in W_m$ to the corresponding domains of $\Tilde{\lambda}\in\bbP^1$. Let $\Tilde{\lambda}\mapsto\Tilde{a}(\Tilde{\lambda})$ be the normalised polynomial whose roots (counted with multiplicities) coincide with the zero set of $\Tilde{A}_1(\Tilde{\lambda}),\ldots,\Tilde{A} _M(\Tilde{\lambda})$ and the roots of $\Tilde{\lambda}\mapsto a\circ\phi^{-1}(\Tilde{\lambda})$ on $\bbP^1\setminus(V_1\cup\ldots\cup V_M)$. Now $\Tilde{\Sigma}=\{(\Tilde{\nu},\Tilde{\lambda})\in\C^2\mid\Tilde{\nu}^2=\Tilde{\lambda}^\aa\Tilde{a}(\Tilde{\lambda})\}$ is a new spectral curve associated to the set of polynomials $\Tilde{A}_1,\ldots,\Tilde{A}_M$. The equations
\begin{equation}
\label{fonctiontilde}
\Tilde{\xi}_m^2=\Tilde{A}_m(\Tilde{\lambda})=\Tilde{A}_m\circ \phi _m ^{-1}  (\Tilde{\lambda})=\Tilde{A}_m (z_m)\quad\hbox{ for }\quad m=1,\ldots M
\end{equation}
define a function $\Tilde{\xi}_m$ on the preimage of $\phi_m(W_m)\cap\bbP^1$ by the map $\Tilde{\lambda}$ into  $\Tilde{\Sigma}$. The function $\Tilde{\xi}_m$ extends to $\Tilde{U}_m=\{p\in\Tilde{\Sigma}\mid\Tilde{\lambda}(p)\in V_m\}$ and coincides on $\{p\in\Tilde{\Sigma}\mid\Tilde{\lambda}(p)\in V_m\setminus(W_1\cap\ldots\cap W_M\}$ with the unpertubed function $\xi_m$. The differential $d\Tilde{\xi}_m$ extends to a meromorphic differential on $\Tilde{\Sigma}$ of the form $\Tilde{\omega}=\frac{\Tilde{b}(\Tilde{\lambda})}{\Tilde{\lambda}^{\bb+\aa\bb}\Tilde{\nu}}d\Tilde{\lambda}$ with a unique polynomial $\Tilde{b}$. The roots of $\Tilde{b}$ are the values of $\Tilde{\lambda}$ at the roots of the derivatives of $\Tilde{A}_1,\ldots,\Tilde{A}_M$. The derivative of (\ref{fonctiontilde}) yields
\begin{align}\label{eq:dA}
2\Tilde{\xi}_m\tfrac{\partial}{\partial\Tilde{\lambda}}\Tilde{\xi}_m&=
\Tilde{A}_m'(z_m(\Tilde{\lambda}))z'_m(\Tilde{\lambda})=\frac{2\Tilde{\xi}_m\Tilde{b}(\Tilde{\lambda})}{\Tilde{\lambda}^{\bb+\aa\bb}\Tilde{\nu}}.
\end{align}
We say that polynomials $\Tilde{A}_1,\ldots,\Tilde{A}_M$ respect the reality condition if the corresponding involution~\eqref{eq:involution 1} or~\eqref{eq:involution 2} or~\eqref{eq:involution 3} lifts to an involution of the new copy of $\bbP^1$ and of $\Tilde{\Sigma}$ such that it acts on $\Tilde{\lambda}$, $\Tilde{\nu}$ and $\Tilde{\omega}$ in the same way as on $\lambda$, $\nu$ and $\omega$, respectively. In this case, the parameter $\Tilde{\lambda}$ is determined up to a real M\"obius transformation (see Section~\ref{sec:moebius}) and $(\Tilde{a},\Tilde{b})\in\calI(a,b)$.
\begin{theorem}\label{smooth param}
Let $(a,b)\in\reg^g$ and let $b/a$ has a first order pole at common roots of $a$ and $b$. The set of polynomials $\Tilde{A}_1,\ldots,\Tilde{A}_M$ with coefficients $\Tilde{a}_{m,2},\ldots,(\Tilde{a}_{m,m}-a_m)$ near zero which respect the reality condition parameterize the quotient of an open neighbourhood of $(a,b)$ in $\calI(a,b)$ modulo the M\"obius transformations described in Section~\ref{sec:moebius}. Together with the $2-\bb$ real parameters introduced in Section~\ref{sec:moebius} they parameterize a submanifold of $\pol^g$, which is an open neighbourhood of $(a,b)$ in $\calI(a,b)$. The tangent space $T_{(a,b)}\calI(a,b)$ is again isomorphic to the polynomials $c$ in Theorem~~\ref{thm:deformation}.
\end{theorem}
\begin{proof}
We have seen that the polynomials $\Tilde{A}_1,\ldots,\Tilde{A}_M$ with coefficients near\-by the coefficients of $A_1,\ldots,A_M$ which respect the reality condition together with a choice of the parameter $\Tilde{\lambda}$ determine $(\Tilde{a},\Tilde{b})\in\calI(a,b)$. With an appropriate choice of $\Tilde{\lambda}$ the roots of $\Tilde{a}$ and $\Tilde{b}$ stay nearby the roots of $a$ and $b$. Conversely, for $(\Tilde{a},\Tilde{b})$ in a sufficiently small neigbourhood of $(a,b)$ in $\calI(a,b)$ the local functions $\Tilde{\xi}_m^2$ take at the roots of $\Tilde{\omega}$ the same values as unique polynomials $\Tilde{A}_1,\ldots,\Tilde{A}_M$. They define a new copy of $\bbP^1$ which is parameterized by $\Tilde{\lambda}$. The roots of $\Tilde{a}$ are the corresponding branch points and the roots of $\Tilde{b}$ coincide with the roots of $\Tilde{\omega}$. Since $\Tilde{\lambda}$ is determined up to the M\"obius transformations described in Section~\ref{sec:moebius} this proves the first statement.

It remains to prove that the map from $\Tilde{A}_1,\ldots,\Tilde{A}_M$ and the $2-\bb$ parameters for the M\"obius transformations described in Section~\ref{sec:moebius} into $\calI(a,b)$ is an immersion into $\pol^g$. First we establish a linear isomorphism from the tangent vectors $\dot{\Tilde{A}}_1,\ldots,\dot{\Tilde{A}}_M$ together with infinitesimal M\"obius transformations described in Section~\ref{sec:moebius} onto the polynomials $\Tilde{c}$ as described in Section~\ref{sec:whitham}. In a second step we show that all these data together uniquely determine the tangent vectors $(\dot{\Tilde{a}},\dot{\Tilde{b}})$. If we differentiate (\ref{fonctiontilde}) with $z_m (\Tilde{\lambda})=\phi^{-1}_m (\Tilde{\lambda})$ we obtain with~\eqref{eq:def c} for the corresponding functions with a tilde and with~\eqref{eq:dA}
\begin{align}
2\Tilde{\xi}_m\tfrac{\partial}{\partial t}\Tilde{\xi}_m&=
\dot{\Tilde{A}}_m(z_m(\Tilde{\lambda}))+\Tilde{A}_m'(z_m(\Tilde{\lambda}))\dot{z}_m(\Tilde{\lambda})=\frac{2\Tilde{\xi}_m\Tilde{c}(\Tilde{\lambda})}{\Tilde{\lambda}^{\aa\bb}\Tilde{\nu}},\nonumber\\\label{eq:A and c}
\frac{\Tilde{\lambda}^\bb\Tilde{c}(\Tilde{\lambda})}{\Tilde{b}(\Tilde{\lambda})}&=
\frac{\dot{\Tilde{A}}_m(z_m(\Tilde{\lambda}))}{\Tilde{A}'_m(z_m(\Tilde{\lambda}))z_m'(\Tilde{\lambda})}+\frac{\dot{z}_m(\Tilde{\lambda})}{z_m'(\Tilde{\lambda})}.
\end{align}
On the right-hand side of~\eqref{eq:A and c} the second term has no poles at the roots of $\Tilde{b}$, since the derivative $z_m'$ of the biholomorphic map $\Tilde{\lambda}\mapsto z_m(\Tilde{\lambda})$ on $W_m$ has no root. Hence the first term on the right hand side of~\eqref{eq:A and c} determines the values of $\Tilde{c}$ at the roots of $\Tilde{b}$ and vice versa. Since the degree of $\Tilde{c}$ differs by $1-\bb$ from the degree of $\Tilde{b}$, these values determines $\Tilde{c}$ up to infintesimal M\"obius transformation described in Section~\ref{sec:moebius}. This shows the first claim.

Equation~\eqref{eq:A and c} determines $\dot{z}_1,\ldots,\dot{z}_M$ in terms of $\dot{\Tilde{A}}_1,\ldots,\dot{\Tilde{A}}_M$ and $\Tilde{c}$. All together they determine the values of $\dot{\Tilde{a}}$ at all roots of $\Tilde{a}$ in $V_1,\ldots,V_M$. The other roots of $\Tilde{a}$ are no roots of $\Tilde{b}$ and the values of $\dot{\Tilde{a}}$ are determined by~\eqref{eq:integrability 1}. Since $\Tilde{a}$ is normalised, this determines $\dot{\Tilde{a}}$ and again with~\eqref{eq:integrability 1} also $\dot{\Tilde{b}}$ in terms of $\dot{\Tilde{a}}$ and $\Tilde{c}$. This proves the second claim.
\end{proof}
\section{ The Whitham flow with common roots of $a$ and $b$}\label{sec:common root}
In this section we preserve the geometric genus along the Whitham flow. For this pupose we assume that $a$ has only simple roots. For general pairs $(\Hat{a},\Hat{b})$ we remove the maximal number of higher order roots and obtain a pair $(a,b)$~\eqref{eq:double roots} with this porperty. If $\resultant(a,b)\ne0$, then Theorem~\ref{thm:deformation} applies. Otherwise the vector fields defined by polynomials $c$ and equation~\eqref{eq:integrability 1} have singularities. We describe how to continuously extend the flow through such singularities. For this purpose we construct an embedding of an open neighbourhood of $(a,b)$ in $\calI(a,b)$ into the corresponding $\calI(\Hat{a},\Hat{b})$ in Theorem~\ref{smooth param}. Common roots of $a$ and $b$ should be considered as higher order roots of $\Hat{a}$. We define  $(\Hat{a},\Hat{b})= (ap^2,bp)$ where $p$ is the polynomial whose roots coincide with roots of $f_O$ counted with multiplicity at common roots of $a$ and $b$. Then  $\Hat{f}=f/p$ has no roots at common roots of $a$ and $b$ and Theorem~\ref{smooth param} applies.

The order $d_m=2\ell_m+1$ of a root of $\Hat{a}$ at a common root of $a$ and $b$ is odd and at least three. On the image of $\calI(a,b)\hookrightarrow\calI(\Hat{a},\Hat{b})$ the geometric genus is preserved together with the number of odd order roots of $\Hat{a}$. This means that those $\Tilde{A}_m$ have exactely one odd order root whose $V_m$ contain a common root of $a$ and $b$. Hence we replace the corresponding parameters~\eqref{eq:deformed pol A} by
\begin{equation}\label{eq:embedding}
\Tilde{A}_m(z_m)=(z_m-2\alpha_{m})p_m^2(z_m)\mbox{ with }p_m(z_m)=z_m^{\ell_m}+\beta_{m,1}z^{\ell_m-1}+\!\ldots\!+\beta_{m,\ell_m}\mbox{ and }\beta_{m,1}=\alpha_{m}.\end{equation}
Since the sum of the zeroes of $\Tilde{A}_m$ is equal to zero, the odd order root $2\alpha_m$ is given by $\beta_{m,1}=\alpha_m$. Due to Theorem~\ref{smooth param} the corresponding elements of $\calI(\Hat{a},\Hat{b})$ depend smoothly on $\beta_{m,1}=\alpha_m$ and $\beta_{m,2},\ldots,\beta_{m,\ell_m}$. By definition of~\eqref{eq:embedding} these elements of $\calI(\Hat{a},\Hat{b})$ are of the form~\eqref{eq:double roots} $(\Tilde{p}^2\Tilde{a},\Tilde{p}\Tilde{b})\in\calI(\Hat{a},\Hat{b})$ with uniquely determined $(\Tilde{a},\Tilde{b})\in\calI(a,b)$ with non-vanishing discriminant $\disc(\Tilde{a})$ and normalised $\Tilde{p}$. Both depend smoothly on the parameters in~\eqref{eq:embedding}. As in Theorem~\ref{smooth param} these parameters describe an open neighbourhood of $(a,b)$ in $\calI(a,b)$ as a subset of $\calI(\Hat{a},\Hat{b})$ which is neither immersed nor a submanifold.
\begin{theorem}\label{commonroots}
Let $(a,b)\in\reg^g\setminus\ct^g$ have non-zero discriminant $\disc(a)\ne 0$ and common roots of $a$ and $b$ away from the fixed point set of~\eqref{eq:involution 1} or \eqref{eq:involution 2}, i.e.\ non-real for~\eqref{eq:involution 1} and non-unital for~\eqref{eq:involution 2}.  Let $c$ be a smooth map from an open neighbourhood $U\subset\pol^g$ of $(a,b)$ into the polynomials of degree $g+1+\aa\bb$ which are fixed points of the corresponding involution~\eqref{eq:involutions c}, such that $c(a,b)$ does not vanish at the common roots of $a$ and $b$. Then there exists for some $\epsilon>0$ a continuous family $(a_t, b_t)_{t\in (-\epsilon,\epsilon)}$ with $(a_0,b_0)=(a,b)$ which is on $(-\epsilon,\epsilon)\setminus\{0\}$ a smooth integral curve of the vector field in Theorem~\ref{thm:deformation} in $\ct^g$.
\end{theorem}
\begin{proof}
We denote the composition of the map from the coefficients of~\eqref{eq:embedding} to $(\Tilde{a},\Tilde{b})\in\calI(a,b)$ with the map $c$ by $\tilde{c}$. Due to Theorem~\ref{smooth param} the product $\Hat{c}=\Tilde{p}\Tilde{c}$ defines a vector field on the image of the map $\calI(a,b)\hookrightarrow\calI(\Hat{a},\Hat{b})$. Since $\Hat{c}$ is a lift of the map $\Tilde{c}$ to the image of $\calI(a,b)\hookrightarrow\calI(\Hat{a},\Hat{b})$, this vector field is tangent to this image. Unfortunately this image is not a submanifold of $\calI(\Hat{a},\Hat{b})$, which is due to Theorem~\ref{smooth param} a manifold. Equation~\eqref{eq:A and c} lifts this vector field to an ODE for the parameters in~\eqref{eq:embedding}. At the initial values of the coefficients in~\eqref{eq:embedding} which are mapped to $(\Hat{a},\Hat{b})$ the map from the parameters in~\eqref{eq:embedding} into $\calI(\Hat{a},\Hat{b})$ is no immersion, and the ODE turns out to be singular.


Next we shall modify the ODE with respect to two different aspects. Firstly we blow up the coordinates in~\eqref{eq:embedding}. Secondly we multiply the vector field with a real function $\renorm$ depending on the blown up coefficients in~\eqref{eq:embedding} in such a way that the new vector field becomes smooth. The new vector field will have a root at the initial value. Due to the second modification the integral curves are reparameterized, but do not change as subsets in $\pol^g$. Due to the first modification the linearisation of the vector field at the root is non-trivial. We want to find two trajectories of the vector field, one moving in and one moving out of the initial singularity. For this purpose we calculate the first derivative of the vector field and find non-trivial stable and unstable eigenspaces. By the Stable Manifold Theorem there then exist integral curves moving in and out of the singularity.

We collect different non-real common roots of $a$ and $b$ and prove that the linearized vector field has non-empty stable and unstable eigenspaces. This means that the first derivative of the vector field at a common root of $a$ and $b$ has non-zero eigenvalues with positive and negative real parts.

First we consider each common root of $a$ and $b$ separately. In this preliminary consideration we pick out one set $V_m$ containing a common root of $a$ and $b$ together with the polynomials $p_m$, $\Tilde{A}_m(z_m)$ and the local parameter $z_m$ as introduced in Section~\ref{sec:param}. In order to simplify notation we drop the index $m$. The following substitution $z\mapsto w=\frac{z}{\alpha}$ and $p(z)\mapsto q(w)$ blows up the coefficients in~\eqref{eq:embedding}. Here $q$ is the unique polynomial with
\begin{equation}\label{eq:def q}
\Tilde{A}(z)=(z-2\alpha)p^2 (z)=\alpha ^{2\ell+1} (w-2)q^2(w)\quad\mbox{with}\quad w=\tfrac{z}{\alpha}.
\end{equation}
As in~\eqref{eq:embedding} $2\alpha\in\C$ is the unique odd order branch point in $V$, and $\ell$ is the degree of the polynomial $q$ whose highest and second highest coefficients are equal to $1$. For $t\to 0$ we have $\alpha\to0$ and $\alpha^\ell q(\frac{z}{\alpha})\to z^\ell$ for bounded $q$. For $\alpha =0$ we have $\Tilde{A}(z)=z^{2\ell+1}$ independently of the polynomial $q$. Hence we can choose an appropriate initial $q$ at $t=0$. Let us calculate the first term of the right-hand side of~\eqref{eq:A and c} in terms of the coefficients of $q$ and $\alpha$:
\begin{align*}
\dot{\Tilde{A}}(z)&=\dot\alpha\alpha^{2\ell}q(\tfrac{z}{\alpha})\left((2\ell\!+\!1)(\tfrac{z}{\alpha}\!-\!2)q(\tfrac{z}{\alpha})-(\tfrac{z}{\alpha})q(\tfrac{z}{\alpha})-2(\tfrac{z}{\alpha})(\tfrac{z}{\alpha}\!-\!2)q'(\tfrac{z}{\alpha})\right)+2\alpha^{2\ell+1}q(\tfrac{z}{\alpha})(\tfrac{z}{\alpha}\!-\!2)\dot{q}(\tfrac{z}{\alpha}),\\
\Tilde{A}' &=\alpha^{2\ell}q(\tfrac{z}{\alpha})\left(q(\tfrac{z}{\alpha})+2(\tfrac{z}{\alpha}-2)q'(\tfrac{z}{\alpha})\right),\\
\frac{\dot{\Tilde{A}}}{\Tilde{A}'}&=\frac{\dot{\alpha}\big((2\ell\!+\!1)(w\!-\!2)q(w)\!-\!wq(w)\!-\!2w(w\!-\!2)q'(w)\big)\!+\!2\alpha(w\!-\!2)\dot{q}(w)}{q(w)+2(w-2)q'(w)}.
\end{align*}
Now $\Tilde{A}'$ has $2\ell$ roots. Besides the $\ell$ double roots of $\Tilde{A}$ it has $\ell$ additional roots, which are equal to the roots of the polynomial
$$\tfrac{\alpha ^\ell}{2\ell+1}\left(q\left(\tfrac{z}{\alpha}\right)+2\left(\tfrac{z}{\alpha}-2\right)q'\left(\tfrac{z}{\alpha}\right)\right).$$
This polynomial has highest coefficient $1$. Locally in $V$, the function $\frac{\Tilde{\lambda}^\bb\Hat{c}(\Tilde\lambda)}{\Hat{b}(\Tilde\lambda)}$ may be uniquely decomposed into a rational function depending on $z\in W\subset\C$, which vanishes as $z\to\infty$, and a holomorphic function $h(z)$ nearby the roots of $\Hat{b}$. Hence there exists a unique polynomial $C(z)$ of degree $\ell-1$, such that in $V\simeq W$, the Laurent decomposition gives
\begin{equation}\label{eq:def C}
\frac{\Tilde{\lambda}^\bb\Hat{c}(\Tilde{\lambda})}{\Hat{b}(\Tilde{\lambda})}=\frac{C(z)}{\frac{\alpha^\ell}{2\ell+1}\left(q\left(\frac{z}{\alpha}\right)+2\left(\frac{z}{\alpha}-2\right)q'\left(\frac{z}{\alpha}\right)\right)}+h(z)\,.
\end{equation}
Therefore we obtain the differential equation in $V_m$
$$\frac{\dot{\Tilde{A}}(z)}{\Tilde{A}'(z)} =\frac{\renorm C(z)}{\frac{\alpha^\ell}{2\ell+1}\left(q\left(\frac{z}{\alpha}\right)+2\left(\frac{z}{\alpha}-2\right)q'\left(\frac{z}{\alpha}\right)\right)}\,.$$
This equation is equivalent to
\begin{gather}\label{eq:ode 2}
\dot{\alpha}\left((2\ell\!+\!1)(w\!-\!2)q(w)-wq(w)-2w(w\!-\!2)q'(w)\right)+2\alpha(w\!-\!2)\dot{q}(w)=\frac{(2\ell\!+\!1)\renorm C(\alpha w)}{\alpha ^\ell}.
\end{gather}
We have $\alpha\to 0$ and $\alpha^\ell q(\frac{z}{\alpha})\to z^\ell$ as $t\to 0$. We should choose the initial $q$ in such a way, that $\frac{\dot{q}}{\dot{\alpha}}$ stays bounded. Otherwise the flow stays in the exceptional locus with $\alpha=0$, all whose points correspond to the same $\Tilde{A}=z^{2\ell+1}$. Consequently, we can neglect for small $t$ the second term of order $\mathbf{O}(\alpha)$ on the left-hand side of~\eqref{eq:ode 2}. Since $c(a,b)$ does not vanish at the common roots of $a$ and $b$, we have $C(0)\neq0$ on the right hand side. Therefore in this preliminary consideration of a single common root of $a$ and $b$ the vector field is multiplied by the function $\renorm=\alpha^{\ell+1}$ and the right-hand side of~\eqref{eq:ode 2} converges in the limit $\alpha\to 0$ to a constant independent of $w$. Consequently for the intial value $\Bar{q}$ of $q$ the following polynomial
\begin{equation}\label{eq:start pol}
(2\ell+1)(w-2)\Bar{q}(w)-w\Bar{q}(w)-2w(w-2)\Bar{q}'(w)
\end{equation}
is a non-vanishing constant. Let us now calculate this initial value $\Bar{q}$:
\begin{lemma}\label{pol}
For each $\ell\in\N$ there exists a unique polynomial $\Bar{q}$ of degree $\ell$ with two highest coefficients $1$, such that the polynomial~\eqref{eq:start pol} equals a constant $K$. This polynomial is the polynomial part of $w^\ell(1-\frac{2}{w})^{-1/2}$.
\end{lemma}
\begin{proof}
For polynomials $\Bar{q}$ of degree $\ell$ with highest coefficient $1$ the polynomial~\eqref{eq:start pol} has degree at most $\ell$. The condition that~\eqref{eq:start pol} is constant yields $\ell-1$ linear equations on the coefficients of $\Bar{q}$. We can solve these equations uniquely by first defining the coefficient of $w^{\ell-1}$ in $\Bar{q}$ such that the coefficient of $w^\ell$ of the polynomial under consideration vanishes and then the lower order coefficients in the inverse order of their power. If we insert for $\Bar{q}(w)=w^\ell (1-\frac{2}{w})^{-1/2}$, then $\Tilde{A}=\alpha ^{2\ell+1} (\frac{z}{\alpha}-2)q^2(\frac{z}{\alpha})$ is equal to $z^{2\ell+1}$ and independent of $\alpha$. Consequently the former expression vanishes. Moreover, if $\Bar{q}(w)$ is the polynomial part of $w^\ell (1-\frac{2}{w})^{-1/2}$, then $\Tilde{A}(z)-z^{2\ell+1}$ is a polynomial of degree $\ell+1$ with respect to $z$ and $w=\frac{z}{\alpha}$. Differentiating with respect to $\alpha$ shows that the polynomial under consideration is constant.
\end{proof}

{\emph{Continuation of the proof of Theorem~\ref{commonroots}}:} For this polynomial $\Bar{q}$ we have
\begin{align*}
(2\ell\!+\!1)(w\!-\!2)\Bar{q}(w)\!-\!w\Bar{q}(w)\!-\!2w(w\!-\!2)\Bar{q}'(w)&=\binom{-\frac{1}{2}}{\ell}(-2)^{\ell+1}(2\ell+1)\\&=\frac{1\cdot3\cdots(2\ell\!-\!1)(2\ell\!+\!1)}{\ell!}(-2).
\end{align*}
At the initial value the ODE takes the form
\begin{align}\label{eq:alpha gamma}
\dot{\alpha}&=\frac{-(2\ell+1)\cdot\ell!\renorm(C(0)+\mathbf{O}(\alpha))}{2\alpha ^\ell1\cdot3\cdots(2\ell-1)(2\ell+1)},&\dot{q}(w)&=\frac{\dot{\alpha}}{2\alpha}(h(w)+\mathbf{O}(\alpha))
\end{align}
where $h(w)$ is the polynomial part of
$$2w\left(q'(w)-\Bar{q}'(w)\right)-(2\ell+1)\left(q(w)-\Bar{q}(w)\right)+\left(1-\tfrac{2}{w}\right)^{-1}(q(w)-\Bar{q}(w))$$
and $\Bar{q}(w)$ is the polynomial in the lemma above. The derivative $\dot{\alpha}$ has to be of order $\mathbf{O}(\alpha)$ in order to compensate the nominator of the right hand side of $\dot{q}$ in~\eqref{eq:alpha gamma}. This confirms $\renorm=\alpha^{\ell+1}$ in order to compensate the power of $\alpha$ in the nominator. Consequently $\dot{\alpha}$ is of order $\mathbf{O}(\alpha)$ and $\dot{q}$ is of order $\mathbf{O}(q-\Bar{q})$ for small $t$. The linearized vector field has block diagonal form with respect to the decomposition of $\alpha$ and $q$. Since we are interested in trajectories on which $\dot{q}/\dot{\alpha}$ is bounded, we restrict to the eigenspaces of the $\dot{\alpha}$ equation. For these eigenspaces $\dot{\alpha}$ is non-zero while $\dot{q}$ vanishes.

It remains to collect all equations corresponding to $m=1,\ldots,M$, and to find eigenvalues of the equation~\eqref{eq:alpha gamma} with non-zero real parts. We decorate the polynomials $C$ in ~\eqref{eq:def C} with the corresponding indices $m=1,\ldots,M$. We end up with equations of the form
$$\dot{\alpha}_m=\frac{\renorm C_m(1+\mathbf{O}(\alpha_1,\ldots,\alpha_M))}{\alpha_m^{\ell_m}}$$
where $\ell_m$ is the degree of $q_m$~\eqref{eq:def q} (with index $m$). Let $N\ge2$ be the least common multiple of $\ell_1+1,\ldots,\ell_M+1$. We blow up again and reparameterize
\begin{align*}
\alpha_1&= e^{\ci\theta_1}s^{\frac{N}{\ell_1+1}},&\alpha_m&=e^{\ci\theta_m}r_m s^{\frac{N}{\ell_m+1}}\mbox{ for }m>1
\end{align*}
with real $s,r_m,\theta_m$  to describe the evolution of $\alpha_1,\ldots,\alpha_M$. Consequently the term $\mathbf{O}(\alpha_1,\ldots,\alpha_m)$ is of order $\mathbf{O}(s^p)$ with $p=\min\{\frac{N}{\ell_1+1},\ldots,\frac{N}{\ell_M+1}\}$. Here we assume that the root of $b$ of index $m=1$ is a common root of $a$ and $b$. With the initial values $\Bar{s}_1=0$ of $s_1$ and $\Bar{\theta}_1$ of $\theta_1$ we have
$$\dot{\alpha}_1=e^{\ci\theta_1}s^{\left(\frac{N}{\ell_1+1}-1\right)}\left(\tfrac{N}{\ell_1+1}\dot{s}+\ci s\dot{\theta}_1\right).$$
Now we choose $\renorm=s^N$. Then we obtain for $m=1$ the equation
$$e^{\ci\theta_1}s^{\left(\frac{N}{\ell_1+1}-1\right)}\left(\tfrac{N}{\ell_1+1}\dot{s}+\ci s\dot{\theta}_1\right)=\frac{C_1s^{\frac{N}{\ell_1+1}}}{e^{\ci\ell_1\theta_1}}(1+\mathbf{O}(s)).$$
This gives the system
\begin{align*}
\dot\theta_1&=\Im\bigl(C_1e^{-\ci(\ell_1+1)\theta_1}\bigr)+\mathbf{O}(s^p),&\dot{s}=&\tfrac{\ell_1+1}{N}s\;\Re\bigl(C_1e^{-\ci(\ell_1+1)\theta_1}\bigr)+\mathbf{O}(s^{p+1}).
\end{align*}
We choose suitable $\Bar{\theta}_1$ to get $\Im( C_1 e^{-\ci(\ell_1+1)\Bar{\theta}_1})=0$. Recall that $C_1\neq 0$ by choice of $c$ having no roots at the roots of $\resultant(a,b)$ in $V_1$. This implies
$$\tfrac{\partial}{\partial\theta_1}\Im\left(C_1e^{-\ci(\ell_1+1)\theta_1}\right)_{|\theta_1=\Bar{\theta}_1}=-(\ell_1+1)\Re\left(C_1e^{-\ci(\ell_1+1)\Bar{\theta}_1}\right)=(\ell_1+1)\Hat{C}_1\neq 0\,.$$
We have to choose the initial value $\Bar{\theta}_1$ at the starting point in the exceptional locus of the blow up in such a way that $\Hat{C}_1=\Re\left(C_1e^{-\ci(\ell_1+1)\Bar{\theta}_1}\right)=\pm|C_1|$ has different signs. Then there exist different solutions with negative and positive eigenvalues of the corresponding linearized equation
\begin{align*}
\dot{\theta}_1&=(\ell_1+1)\Hat{C}_1\theta_1\,,&\dot{s}&=\tfrac{\ell_1+1}{N}s\;\Re\,\bigl(C_1 e^{-(\ell_1+1)\Bar{\theta}_1}\bigr)=-\tfrac{(\ell_1+1)\Hat{C}_1}{N}s.
\end{align*}
For $m >1$ we reparameterize $\alpha_m$ by the parameters $(s,r_m,\theta_m)$ and get with the initial values $(0,\Bar{r}_m,\Bar{\theta}_m)$ with $\Bar{r}_m\ne0$ the equation
\begin{align*}
\dot{\alpha}_m&=\left(\ci r_m\dot{\theta}_m s+\dot{r}_ms+\tfrac{N}{\ell_m+1}r_m\dot{s}\right)e^{\ci\theta_m}s^{\frac{N}{\ell_1+1}-1}\\
&=\frac{C_ms^N(1+\mathbf{O}(s^p))}{s^{\frac{\ell_m N}{\ell_m+1}}e^{\ci\ell_m\theta_m}r_m^{\ell_m}}=C_m e^{-\ci\ell_m\theta_m}r_m^{-\ell_m}s^{\frac{N}{\ell_m+1}}(1+\mathbf{O}(s^p)).
\end{align*}
This implies
$$\ci r_m\dot{\theta}_ms+\dot{r}_ms+\tfrac{N}{\ell_m+1}r_m\dot{s}=C_me^{-\ci(\ell_m+1)\theta_m}r_m^{-\ell_m}s(1+\mathbf{O}(s^p))\,.$$
Hence we study the system
\begin{align*}
\dot{\theta}_m&=\Im\left(C_m e^{-\ci(\ell_m+1)\theta_m}\right)r_m^{-(\ell_m+1)}+\mathbf{O}(s^p)\\
\dot{r}_m &=\Re\left(C_m e^{-\ci(\ell_m +1)\theta_m}\right)r_m^{-\ell_m}-\tfrac{\ell_1+1}{\ell_m+1}\Re\left(C_1e^{-\ci (\ell_1+1)\theta_1}\right)r_m +\mathbf{O}(s^p).
\end{align*}
Now we choose $\Bar{\theta}_m$ in such a way that $\Im(C_me^{-\ci(\ell_m+1)\Bar{\theta}_m})=0$ and $\Hat{C}_m=-\Re(C_me^{-\ci (\ell_m+1)\Bar{\theta}_m})=\pm|C_m|$ has the same sign as $\Hat{C}_1$: For a choice of $\Bar{\theta}_1$, and then a sign for $\Hat{C}_1$, we fix a choice of $(\Bar{\theta}_2,\ldots,\Bar{\theta}_M)$ in such a way that $\Hat{C}_1,\ldots,\Hat{C}_M$ have the same signs. Finally, we choose $\Bar{r}_m>0$ which satisfy
$$\Re\left(C_me^{-\ci(\ell_m+1)\Bar{\theta}_m}\right)\Bar{r}_m^{-(\ell_m+1)}-\tfrac{\ell_1+1}{\ell_m+1}\Re\left(C_1e^{-\ci(\ell_1+1)\Bar{\theta}_1}\right)=0.$$
The linearized system is for $m=2,\ldots,M$
\begin{align*}
\dot{\theta}_m&=(\ell_m+1)\Hat{C}_m\Bar{r}_m^{-(\ell_m+1)}\theta_m\,,&\dot{r}_m&=\tfrac{(\ell_m+1)}{r^{\ell_m+1}}\Hat{C}_m(r_m-\Bar{r}_m).
\end{align*}
This shows the existence of initial values with different signs of the eigenvalues for the linearized system. The Stable Manifold Theorem~\cite[Theorem~9.3]{teschl2012} guarantees the existence of trajectories moving in and out of the singularity. Since the exponent $N$ of the factor $\Gamma=s^N$ is at least $2$, the trajectory of the original vector field moves in and out in finite time.
\end{proof}
\begin{remark}
We assume that the common roots of $a$ and $b$ are non-real, i.e.\ not fixed points of the involution~\eqref{eq:involution 1} or \eqref{eq:involution 2}. At real common roots the coefficients of $A_m$~\eqref{eq:pol A}, $\Tilde{A}_m$~\eqref{eq:embedding} and of $\Bar{q}$ in Lemma~\ref{pol} are real. In this case it can happen that there only exist integral curves moving in or out of the singularity after multiplying the vector field induced by $c$ with $\pm1$.
\end{remark}
%

%
%
%
%
%
\end{document}